\documentclass{article}

\usepackage{graphicx}
\usepackage{color}
\usepackage{amsmath,amssymb,amsthm,amsrefs,amsfonts,mathrsfs,indentfirst}
\usepackage[all]{xy}
\usepackage{mathtools}
\usepackage{hyperref}

\def\co{\colon\thinspace}
\DeclareMathAlphabet{\mathsfsl}{OT1}{cmss}{m}{sl}

\def\spinc{\mathrm{Spin}^c}
\def\relspinc{\underline{\mathrm{Spin}}^c}
\newcommand{\Bslash}{\backslash\hspace{-3pt}\backslash}

\def\ond{\nu^{\circ}}

\newtheorem{thm}{Theorem}[section]
\newtheorem{lem}[thm]{Lemma}
\newtheorem{cor}[thm]{Corollary}
\newtheorem{prop}[thm]{Proposition}
\newtheorem{conj}[thm]{Conjecture}

\theoremstyle{definition}
\newtheorem{defn}[thm]{Definition}
\newtheorem{rem}[thm]{Remark}

\begin{document}

\title{Property G and the $4$--genus}

\author{{\Large Yi NI}\\{\normalsize Department of Mathematics, Caltech, MC 253-37}\\
{\normalsize 1200 E California Blvd, Pasadena, CA
91125}\\{\small\it Email\/:\quad\rm yini@caltech.edu}}

\date{}

\maketitle

\begin{abstract}
We say a null-homologous knot $K$ in a $3$--manifold $Y$ has Property G, if the Thurston norm and fiberedness of the complement of $K$ is preserved under the zero surgery on $K$.
In this paper, we will show that, if the smooth $4$--genus of $K\times\{0\}$ (in a certain homology class) in $(Y\times[0,1])\#N\overline{\mathbb CP^2}$, where $Y$ is a rational homology sphere, is smaller than the Seifert genus of $K$, then $K$ has Property G. When the smooth $4$--genus is $0$, $Y$ can be taken to be any closed, oriented $3$--manifold.
\end{abstract}

\

\section{Introduction}

 A general theme in the study of Dehn surgery is to get the information about the topology of Dehn surgery from the topology of the original knot. In this paper, we will consider the following situation.
If $F$ is a Seifert surface for a null-homologous knot $K\subset Y$, then there is a closed surface $\widehat F$ in the zero surgery on $K$ obtained by capping off $\partial F$ with a disk.  
Suppose that we know a certain property of $F$ as a subsurface of $Y\setminus K$, can we deduce similar properties for $\widehat F$?

Before we explain this problem in more detail, let us establish some notations we will use.
If $Z$ is a submanifold of a manifold $Y$, let $\nu(Z)$ be a closed tubular neighborhood of $Z$ in $Y$, let $\ond(Z)$ be the interior of $\nu(Z)$, and let $Y\Bslash Z=Y\backslash \ond(Z)$.
Given a null-homologous knot $K$ in a $3$--manifold $Y$, and $\frac pq\in\mathbb Q\cup\{\infty\}$, let $Y_{\frac pq}(K)$ be the manifold obtained by $\frac pq$--surgery on $K$.

All manifolds are smooth and oriented unless otherwise stated.

\begin{defn}
Suppose that $K$ is a null-homologous knot in a closed $3$--manifold
$Y$. A compact surface $F\subset Y$ is a {\it Seifert-like
surface} for $K$, if $\partial F=K$. The homology class $[F]\in H_2(Y,K)$ is called a {\it Seifert homology class}. 
When $F$ is connected, we say
that $F$ is a {\it Seifert surface} for $K$. We also view a
Seifert-like surface as a properly embedded surface in $Y\Bslash K$. 
\end{defn}

\begin{defn}
Suppose that $M$ is a compact $3$--manifold, a properly embedded
surface $S\subset M$ is {\it taut} if $\chi_-(S)=\chi_-([S])$ in
$H_2(M,\partial M)$, no proper subsurface of
$S$ is null-homologous, and for any non-sphere component $S_i$ of $S$, $[S_i]$ is not represented by a sphere. Here $\chi_-(\cdot)$ is the Thurston norm, see Subsection~\ref{subsect:ThurstonNorm}.
\end{defn}

\begin{defn}\label{defn:PropG}
Suppose that $K$ is a null-homologous knot in a closed $3$--manifold
$Y$, $\varphi\in H_2(Y,K)$ is the homology class of a Seifert-like surface. We say $K$ has {\it Property G} with respect to $\varphi$, if the following conditions
hold:

\noindent(G1) If $F$ is a taut Seifert-like surface with $[F]=\varphi$, then $\widehat F$ is 
taut in $Y_0(K)$;

\noindent(G2) if $Y_0(K)$ fibers over $S^1$, such that the
homology class of the fiber is the natural extension $\widehat{\varphi}$ of $\varphi$, then $K$ is a fibered knot, and
the homology class of the fiber is $\varphi$.

If the first (or second) condition holds, then we say that $K$ has
{\it Property G1 ({\rm or} G2)} with respect to $\varphi$. If $K$ has Property G ({or} G1, G2) with respect to every Seifert homology class, then we say $K$ has {\it Property G ({\rm or} G1, G2)}.
\end{defn}

Gabai \cite{G3} proved that knots in $S^3$ have Property G. As he remarked in \cite{G3}, the proof also works when $Y$ is reducible, $Y-K$ is irreducible and $H_1(Y)$ is torsion-free. It is not hard to see connected sums of knots have Property G. In \cite{NiFibred,AiNi}, it is showed that
null-homologous knots in L-spaces have Property G2, and the same argument also implies that such knots have Property G1. In \cite{NiNSSphere}, we proved that if $Y$ contains a non-separating sphere and $Y-K$ is irreducible, then $K$ has Property G. 

\begin{rem}
The reason that we want to use Seifert-like surfaces instead of Seifert surfaces in Definition~\ref{defn:PropG} is that if $b_1(Y)>0$, not every Seifert homology class is represented by a taut Seifert surface. For example, let $Y=Y_1\#Y_2$, $K\subset Y_1$, and $Y_2$ is irreducible. If a Seifert homology class $\varphi\in H_2(Y,K)\cong H_2(Y_1,K)\oplus H_2(Y_2)$ has a nonzero component in $H_2(Y_2)$, then any taut surface representing $\varphi$ has to be disconnected.
\end{rem}

\begin{thm}\label{thm:4genus}
Let $Y$ be a closed, oriented, connected $3$--manifold, and let $K\subset Y$ be a null-homologous knot.  
Suppose that $F$ is a taut Seifert surface for $K$. Let $X=(Y\times[0,1])\#N\overline{\mathbb CP^2}$ for some $N$, and let \[\iota\co (Y,K)\to (X,K\times\{0\})\]
be the inclusion map.
If there exists a properly embedded smooth connected surface $G\subset X$ with $\partial G=K\times\{0\}$, $[G]=\iota_*[F]\in H_2(X,K\times\{0\})$, and $g(G)<g(F)$,
then $K$ has Property G with respect to $[F]$.
\end{thm}

The conclusion of Theorem~\ref{thm:4genus} is just about Property G with respect to a special homology class $[F]$, and $F$ is indeed a Seifert surface. When $Y$ is a rational homology sphere, there is only one such homology class, and any taut Seifert-like surface must be connected, so we have the following special case which is worth mentioning.

\begin{cor}\label{cor:Rational}
Let $Y$ be a closed connected $3$--manifold with $b_1(Y)=0$, and let $K\subset Y$ be a null-homologous knot. If the smooth $4$--genus of $K\times\{0\}$ in $X=(Y\times[0,1])\#N\overline{\mathbb CP^2}$, with the homology class of the surface in the image of $\iota_*$, is smaller than the Seifert genus of $K$ for some $N$, then $K$ has Property G.
\end{cor}

If the above $4$--genus is $0$, then the conclusion of Theorem~\ref{thm:4genus} can be strengthened to assert that $K$ has Property G even when $b_1(Y)>0$.

\begin{thm}\label{thm:Slice}
Let $Y$ be a closed, oriented $3$--manifold, and let $K\subset Y$ be a knot. If $K\times\{0\}$ bounds a properly embedded smooth disk $G$ in $X=(Y\times[0,1])\#N\overline{\mathbb CP^2}$, with $[G]$ in the image of $\iota_*$, then $K$ has Property G.  
\end{thm}

Since Property G does not depend on the orientation of the manifolds involved, one can replace the $X$ in Theorems~\ref{thm:4genus} and~\ref{thm:Slice} with $(Y\times[0,1])\#N{\mathbb CP^2}$, the conclusions of these two theorems are still true. However, our current methods break down if $X$ is replaced with $(Y\times[0,1])\#N(\mathbb CP^2\#\overline{\mathbb CP^2})$. 

\begin{conj}\label{conj:Immerse}
The conclusion of Theorem~\ref{thm:4genus} still holds if $X$ is replaced with $(Y\times[0,1])\#N(\mathbb CP^2\#\overline{\mathbb CP^2})$.
\end{conj}

If $K$ bounds an immersed surface in $Y$, then the immersed surface can be perturbed to an immersed surface with normal self-intersections in $Y\times[0,1]$. One can then resolve the self-intersections to get an embedded surface in $(Y\times[0,1])\#N(\mathbb CP^2\#\overline{\mathbb CP^2})$ for some $N$ so that the homology class of the embedded surface is in the image of $\iota_*$. Thus the condition in Conjecture~\ref{conj:Immerse} essentially says that the ``immersed genus'' of $K$ is smaller than the Seifert genus in the relative homology class. Similar to Theorem~\ref{thm:Slice}, one can hope that if the ``immersed genus'' of $K$ is zero, then $K$ has Property G.

\begin{conj}\label{conj:Nullhomo}
Null-homotopic knots have Property G.
\end{conj}

Boileau has conjectured that null-homotopic knots have Property G2 \cite[Problem~1.80]{Kirby}\footnote{
There are trivial counterexamples to the original (slightly different) question of Boileau. See the remark after \cite[Question~1.3]{NiFibredSurg}.
}. The question whether null-homotopic knots have Property G1 was asked to the author by David Gabai. 

A special case of Conjecture~\ref{conj:Nullhomo} is the following assertion: If $K\subset Y$ is a nontrivial null-homotopic knot, $F$ is a Seifert surface, then $[\widehat F]$  is not represented by a sphere. This special case was proved by Hom and Lidman \cite{HL} in the case when $Y$ is a prime rational homology sphere. In fact, using a result due to Gabai \cite{G2} and Lackenby \cite{Lackenby}, it is straightforward to prove the case $b_1(Y)>0$.

\begin{prop}\label{prop:b1>0}
Let $Y$ be a closed, oriented, connected $3$--manifold with $b_1(Y)>0$, and $K\subset Y$ be a nontrivial null-homotopic knot such that $Y-K$ is irreducible, then $Y_0(K)$ does not have an $S^1\times S^2$ connected summand.
\end{prop}

\begin{rem}
Hom and Lidman \cite{HL} also proved that if $b_1(Y)=0$ and $K\subset Y$ is a null-homologous knot, then $Y_0(K)$ is not homeomorphic to $Y\#S^1\times S^2$. Using the same argument as in their paper combined with Theorem~\ref{thm:ZeroMCTwisted}, we can remove the $b_1(Y)=0$ condition in their paper. In a later paper \cite{NiPropR}, the author removed the $b_1(Y)>0$ condition in Proposition~\ref{prop:b1>0}.
\end{rem}

Using the argument in our paper, it is not hard to prove Property G for null-homologous knots in many $3$--manifolds with ``simple'' Heegaard Floer homology. For example, knots in the Brieskorn sphere $\Sigma(2,3,7)$ have Property G.
We will only do this for torus bundles.

\begin{thm}\label{thm:TB}
Let $Y$ be a $T^2$--bundle over $S^1$, then
null-homologous knots in $Y$ have Property G.
\end{thm}

Given what is known about Property G, it is reasonable to expect the following conjecture to be true.

\begin{conj}
If $Y$ is a closed, oriented, connected, reducible $3$--manifold, $K\subset Y$ is a null-homologous knot with $Y\Bslash K$ being irreducible, then $K$ has Property G.
\end{conj}

This paper is organized as follows. In Section~\ref{sect:Prelim}, we recall some basic material about Heegaard Floer homology.
In Section~\ref{sect:ZeroMC}, we state and prove the general case of the zero surgery formula of Heegaard Floer homology from \cite{OSzIntSurg}.
In Section~\ref{sect:ZeroSurg}, we use a standard argument to prove a rank inequality relating
$\widehat{HFK}(Y,K)$ and $HF^+(Y_0(K))$. Our main results follow from this rank inequality. We will also prove Proposition~\ref{prop:b1>0}. In Section~\ref{sect:TB}, we prove Theorem~\ref{thm:TB}. In the Appendix, we sketch a proof of the folklore Theorem~\ref{thm:OriginalMC}, which has been frequently cited in the literature.

\

\noindent{\bf Acknowledgements.} This
research was funded by NSF grant numbers DMS-1252992 and DMS-1811900. We are grateful to the referee for the detailed comments which helped to improve the exposition of this paper.


\section{Preliminaries on Heegaard Floer homology}\label{sect:Prelim}

In this section, we will collect some results we need on Heegaard Floer homology.

\subsection{Different versions of Heegaard Floer homology}

Heegaard Floer homology \cite{OSzAnn1}, in its most fundamental form, assigns a package of invariants
$$\widehat{HF}, HF^+, HF^-, HF^{\infty}, HF_{\mathrm{red}}$$
to a closed, connected, oriented $3$--manifold $Y$ equipped with a Spin$^c$ structure $\mathfrak s\in\spinc(Y)$.
We often use $HF^{\circ}$ to denote these invariants, where $\circ$ is one of $\hat{\ },+,-,\infty,\mathrm{red}$.

There is a short exact sequence of chain complexes
$$
\xymatrix{
0\ar[r] &CF^-(Y,\mathfrak s)\ar[r]^-{\eta} &CF^{\infty}(Y,\mathfrak s)\ar[r]^{\rho} &CF^+(Y,\mathfrak s)\ar[r] &0,
}
$$
which induces an exact triangle relating $HF^+, HF^-, HF^{\infty}$. 
When $c_1(\mathfrak s)$ is nontorsion, as shown in \cite[Corollary~2.4]{OSzSympl}, $\rho_*=0$. Hence we have a natural short exact sequence
\begin{equation}\label{eq:SES+-inf}
\xymatrix{
0\ar[r] &HF^+(Y,\mathfrak s)\ar[r] &HF^-(Y,\mathfrak s)\ar[r]^-{\eta_*} &HF^{\infty}(Y,\mathfrak s)\ar[r] &0.
}
\end{equation}


\subsection{The cobordism map}

Suppose that $W\co Y_1\to Y_2$ is an oriented cobordism, with $Y_1,Y_2$ connected. Given $\mathfrak S\in\spinc(W)$, there is a homomorphism
$$F^{\circ}_{W,\mathfrak S}\co HF^{\circ}(Y_1,\mathfrak S|Y_1)\to HF^{\circ}(Y_2,\mathfrak S|Y_2),$$
induced by a chain map
$$f^{\circ}_{W,\mathfrak S}\co CF^{\circ}(Y_1,\mathfrak S|Y_1)\to CF^{\circ}(Y_2,\mathfrak S|Y_2).$$

Suppose that $W$ is the composition of two cobordisms $$W_1\co Y_1\to Y_3,\quad W_2\co Y_3\to Y_2.$$ Given $\mathfrak S_1\in\spinc(W_1),\mathfrak S_2\in\spinc(W_2)$ with $\mathfrak S_1|Y_3=\mathfrak S_2|Y_3$, we have the composition law \cite[Theorem~3.4]{OSzFour}
\begin{equation}\label{eq:CobCompos}
F^{\circ}_{W_2,\mathfrak S_2}\circ F^{\circ}_{W_1,\mathfrak S_1}=\sum_{\mathfrak S\in\spinc(W),\mathfrak S|W_i=\mathfrak S_i}F^{\circ}_{W,\mathfrak S}.
\end{equation}

The composition law becomes simpler in the following special case.

\begin{lem}\label{lem:SpincInj}
Suppose that $W_1$ is obtained from $Y_1\times[0,1]$ by adding $b$ one-handles to $Y_1\times1$ and then $1$ two-handle, where the attaching curve of the two-handle is a null-homologous knot in $Y_1\#b(S^1\times S^2)$,
and the framing is nonzero. Then for any $\mathfrak S\in\spinc(W)$, we have
$$
F^{\circ}_{W_2,\mathfrak S|W_2}\circ F^{\circ}_{W_1,\mathfrak S|W_1}=F^{\circ}_{W,\mathfrak S}.
$$ 
\end{lem}
\begin{proof}
Consider the exact sequence $$
\xymatrix{
H^1(W_1)\ar[r]^-{\eta^*} &H^1(Y_3)\ar[r] &H^2(W_1,Y_3)\ar[r]^-{\rho^*}&H^2(W_1).
}
$$
By the Poincar\'e--Lefschetz duality, the map $\rho^*$ can be identified with $$H_2(W_1,Y_1)(\cong\mathbb Z)\to H_2(W_1,\partial W_1),$$
which is injective by our condition on $W_1$. Hence $\eta^*$ is surjective.

Now consider the exact sequence $$
\xymatrix{
H^1(W_1)\oplus H^1(W_2)\ar[r]^-{\iota^*} &H^1(Y_3)\ar[r]&H^2(W)\ar[r]^-{\pi^*}&H^2(W_1)\oplus H^2(W_2).
}
$$
As $\eta^*$ is surjective, $\iota^*$ is also surjective, so $\pi^*$ is injective, hence the restriction map
$$\spinc(W)\to\spinc(W_1)\times\spinc(W_2)$$ is injective. Our conclusion then follows from (\ref{eq:CobCompos}).
\end{proof}

The following adjunction relation proved in \cite{OSzSympl} plays an important role in our paper.

\begin{thm}\label{thm:AdjRel}
For every genus $g$, there is an element $\zeta\in \mathbb Z[U]\otimes_\mathbb Z\Lambda^*H_1(\Sigma)$ of degree $2g$ with the following property. Given any smooth, oriented, $4$--dimensional cobordism $W$ from $Y_1$ to $Y_2$ (both of which are connected $3$--manifolds), any smoothly embedded, connected, oriented submanifold $\Sigma\subset W$ of genus $g$, and any $\mathfrak S\in\spinc(W)$ satisfying that
\[
\langle c_1(\mathfrak S),[\Sigma]\rangle-[\Sigma]\cdot [\Sigma]=-2g(\Sigma), 
\]
we have the relation
\[
F^{\circ}_{W,\mathfrak S}=F^{\circ}_{W,\mathfrak S+\epsilon\mathrm{PD}[\Sigma]}(i_*(\zeta)),
\]
where $\epsilon$ is the sign of $\langle c_1(\mathfrak S),[\Sigma]\rangle$, and \[i_*\co \mathbb Z[U]\otimes_{\mathbb Z}\Lambda^*H_1(\Sigma)\to \mathbb Z[U]\otimes_{\mathbb Z}\Lambda^*(H_1(W)/\mathrm{Tors})\] is the map induced by the inclusion $i\co \Sigma\to W$.
\end{thm}


\subsection{The knot Floer chain complex}\label{Subsect:KnotFloer}

In this subsection, we will briefly recall the construction of the knot Floer chain complex from \cite{OSzKnot}. (See also \cite{RasThesis}.) We will then discuss its relationship with large Dehn surgery.

Suppose that $K$ is a null-homologous knot in $Y$. As in \cite{OSzLink}, any relative Spin$^c$ structure $\xi$ on $X=Y\Bslash K$ is represented by a vector field $v$ on $X$ whose restriction to $\partial X$ is a canonical boundary parallel vector field.  Let $c_1(\xi)\in H^2(Y,K)\cong H^2(X,\partial X)$ be the obstruction to extending the boundary-induced trivialization of $v^{\bot}|(\partial X)$ to a trivialization of $v^{\bot}$ over $X$.

Suppose that
$$(\Sigma,\mbox{\boldmath${\alpha}$},\mbox{\boldmath${\beta}$},w,z)$$
is a doubly pointed Heegaard diagram for  $(Y,K)$.
Given $\mathfrak s\in\spinc(Y)$ and $\xi\in\relspinc(Y,K)$ such that $\mathfrak s$ extends $\xi$, let $CFK^{\infty}(Y,K,\xi)$ be the knot Floer chain complex of $(Y,K,\mathfrak s)$. It is generated by $[\mathbf x,i,j]$ satisfying
\begin{equation}\label{eq:ijfiltration}
\underline{\mathfrak s}_{w,z}(\mathbf x)+(i-j)\mathrm{PD}[\mu]=\xi,
\end{equation}
where
$\mathbf x\in \mathbb T_{\alpha}\cap \mathbb T_{\beta}$. The differential $\partial$ is given by
$$\partial[\mathbf x,i,j]=\sum_{\mathbf y\in \mathbb T_{\alpha}\cap \mathbb T_{\beta}}\sum_{\phi\in\pi_2(\mathbf x,\mathbf y),\mu(\phi)=1}\#\widehat{\mathcal M}(\phi)[\mathbf y,i-n_w(\phi),j-n_z(\phi)].$$
The pair $(i,j)$ defines a double filtration on $CFK^{\infty}(Y,K,\xi)$.
Define the knot Floer homology
$$\widehat{HFK}(Y,K,\xi)=H_*(CFK^{\infty}(Y,K,\xi)\{i=j=0\}).$$

Given a homology class $\varphi\in H_2(Y,K)$ represented by a Seifert surface $F$ for $K$,
let $\widehat{\varphi}\in H_2(Y_0(K))$ be the homology class which extends $\varphi$.
Given $\mathfrak s\in\spinc(Y)$ and $k\in\mathbb Z$, let $\xi_k\in\relspinc(Y,K)$ be the extension of $\mathfrak s$ so that 
\begin{equation}\label{eq:xi_k}
\langle c_1(\xi_k),{\varphi}\rangle=2k+1,
\end{equation}
and let $\mathfrak t_k\in\spinc(Y_0(K))$ be the Spin$^c$ structure satisfying that $\mathfrak s|(Y\Bslash K)=\mathfrak t_k|(Y\Bslash K)$ and
$$\langle c_1(\mathfrak t_k),\widehat{\varphi}\rangle=2k.$$
Denote
$$C=CFK^{\infty}(Y,K,\mathfrak s,\varphi):=CFK^{\infty}(Y,K,\xi_0).$$
Note that the $(i,j)$--filtration on $C$ depends on the homology class $\varphi$ through (\ref{eq:ijfiltration}), because the definition of $\xi_0$ (which is (\ref{eq:xi_k})) involves $\varphi$. For simplicity, we often suppress $\varphi$ in this paper.
Let
\begin{eqnarray*}
\widehat{HFK}(Y,K,\mathfrak s,\varphi,k)&=&\widehat{HFK}(Y,K,\xi_k),\\
\widehat{HFK}(Y,K,\varphi,k)&=&\bigoplus_{\mathfrak s\in\spinc(Y)}\widehat{HFK}(Y,K,\mathfrak s,\varphi,k).
\end{eqnarray*}

\begin{rem}
We often abuse the notation by letting
\[
C=CFK^{\infty}(Y,K,\varphi)=\bigoplus_{\mathfrak s\in\spinc(Y)}CFK^{\infty}(Y,K,\mathfrak s,\varphi)
\]
if a Spin$^c$ structure $\mathfrak s\in\spinc(Y)$ is not explicitly mentioned in the context.
\end{rem}

There are chain complexes
$$A^+_k=C\{i\ge0 \text{ or }j\ge k\},\quad k\in\mathbb Z$$
and $B^+=C\{i\ge0\}\cong CF^+(Y,\mathfrak s)$. As in \cite{OSzIntSurg}, there are chain maps
$$v^+_k,h^+_k\co A^+_k\to B^+.$$
Here $v^+_k$ is the vertical projection, and $h^+_k$ first projects $A^+_k$ to $C\{j\ge k\}$, then maps to $C\{j\ge0\}$ via $U^k$, finally maps to $B^+$ via a chain homotopy equivalence $C\{j\ge0\}\to C\{i\ge0\}$.

\begin{rem}\label{rem:canonical}
The chain homotopy equivalence $C\{j\ge0\}\to C\{i\ge0\}$ is obtained by changing the basepoint from $z$ to $w$. In \cite{OSzIntSurg}, Ozsv\'ath and Szab\'o said that the chain homotopy equivalence is canonical up to sign. This assertion is justified by \cite[Theorem~2.1]{OSzFour} and \cite{Gartner}. (If we use $\mathbb Z/2\mathbb Z$ coefficients, which are enough for the applications in our paper, we can also refer to \cite{JTZ}.)
\end{rem}

There is a commutative diagram
$$
\xymatrix{
A^+_{k+1}\ar[rd]^-{v^+_{k+1}}&\\
A^+_{k}\ar[r]^-{v^+_{k}}\ar[u]&B^+,
}
$$
where the vertical arrow denotes the natural quotient map. Let $$(v^+_k)_*,(h^+_k)_*\co H_*(A^+_k)\to H_*(B^+)$$ be the induced map on homology, then the previous commutative diagram implies that
\begin{equation}\label{eq:vIm}
\mathrm{im}(v^+_k)_*\subset \mathrm{im}(v^+_{k+1})_*.
\end{equation}
Similarly,
\begin{equation}\label{eq:hIm}
\mathrm{im}(h^+_k)_*\supset \mathrm{im}(h^+_{k+1})_*.
\end{equation}

The following theorem is contained in \cite{OSzKnot} and \cite{OSzIntSurg}.
Let $F$ be a Seifert surface for $K$ in the homology class $\varphi$. Given $n>0$ and $t\in\mathbb Z/n\mathbb Z$, let $W'_n(K)\co Y_n(K)\to Y$ be the natural two-handle cobordism, and let $\widehat F_n\subset W'_n(K)$ be the closed surface obtained by capping off $\partial F$ with a disk. Let $\widehat{\varphi}_n=[\widehat F_n]\in H_2(W'_n(K))$.
Let $k$ be an integer satisfying $k\equiv t\pmod n$ and $|k|\le\frac n2$.
Let $\mathfrak x_k,\mathfrak y_k\in\spinc(W'_n(K))$ satisfy $\mathfrak x_k|Y=\mathfrak y_k|Y=\mathfrak s$ and
\begin{equation}\label{eq:xkyk}
\langle c_1(\mathfrak x_k),\widehat{\varphi}_n\rangle=2k-n\qquad\langle c_1(\mathfrak y_k),\widehat{\varphi}_n\rangle=2k+n,
\end{equation}
and let $\mathfrak s_t=\mathfrak x_k|Y_n(K)=\mathfrak y_k|Y_n(K)$.

\begin{thm}\label{thm:LargeSurg}
When $n\ge2g(F)$, the chain complex
$CF^+(Y_n(K),\mathfrak s_t)$ is represented by the chain complex $A^+_k$, in the sense that there is an isomorphism $$\Psi^+_{n}\co CF^+(Y_n(K),\mathfrak s_t)\to A^+_k.$$
Moreover, the following squares commute:
$$
\xymatrixcolsep{3pc}\xymatrix{
CF^+(Y_n(K),\mathfrak s_t)\ar[r]^-{f^+_{W'_n(K),\mathfrak x_k}}\ar[d]^-{\Psi^+_{n}}&CF^+(Y,\mathfrak s)\ar[d]^-{=}\\
A^+_k \ar[r]^-{v^+_k} &B^+, 
}\quad
\xymatrix{
CF^+(Y_n(K),\mathfrak s_t)\ar[r]^-{f^+_{W'_n(K),\mathfrak y_k}}\ar[d]^-{\Psi^+_{n}}&CF^+(Y,\mathfrak s)\ar[d]^-{=}\\
A^+_k \ar[r]^-{h^+_k} &B^+.
}
$$
\end{thm}


\subsection{The Thurston norm of $Y\Bslash K$ and $Y_0(K)$}\label{subsect:ThurstonNorm}

In this subsection, we will recall a few facts about the Thurston norm \cite{Thurston} and Heegaard Floer homology.

\begin{defn}
Let $S$ be a compact oriented surface with connected components
$S_1,\dots,S_n$. We define
$$\chi_-(S)=\sum_i\max\{0,-\chi(S_i)\}.$$
Let $M$ be a compact oriented 3--manifold, $A$ be a compact
codimension--0 submanifold of $\partial M$. Let
$h\in H_2(M,A)$. The {\it Thurston norm} $\chi_-(h)$ of $h$ is
defined to be the minimal value of $\chi_-(S)$, where $S$ runs over all
the properly embedded surfaces in $M$ with $\partial S\subset A$,
and $[S]=h$.
\end{defn}

The following properties of Heegaard Floer homology are well-known, see \cite{OSzGenus,Gh,NiFibred,NiNormCos,NiIntHier}.

\begin{thm}\label{thm:ThurstonNorm}
Suppose that $Y$ is a closed 3--manifold.
Let $G$ be a taut surface in $Y$. Then
$$
\chi_-(G)=\max\left\{\langle c_1(\mathfrak s),[G]\rangle\left| \mathfrak s\in \spinc(Y),\:HF^+(Y,\mathfrak s)\ne0\right.\right\}
$$
\end{thm}

\begin{thm}\label{thm:SeifertNorm}
Suppose that $K$ is a null-homologous knot in a closed 3--manifold $Y$.
Let $F$ be a taut Seifert-like surface for $K$. Then
$$\chi_-(F)+2=\max\left\{\langle c_1(\xi),[F]\rangle\left|\xi\in\relspinc(Y\Bslash K), \widehat{HFK}(Y,K,\xi)\ne0\right.\right\}.$$
\end{thm}

\begin{thm}\label{thm:KnotFibered}
Suppose that $K$ is a null-homologous knot in a 3--manifold $Y$ with a genus
$g$ Seifert surface $F$, and $Y\Bslash K$ is irreducible. If
$$\widehat{HFK}(Y,K,[F],g(F))\cong\mathbb Z,$$ then $K$ is fibered with fiber $F$.
\end{thm}

The above three theorems hold true even when we use field coefficients.


\section{Computing $HF^+(Y_0(K))$ via mapping cone}\label{sect:ZeroMC}

 In this section, we present detailed proofs for the general case of two theorems mentioned in \cite[Subsection~4.8]{OSzIntSurg}. The original theorems of Ozsv\'ath and Szab\'o are the (untwisted and twisted) zero surgery formula for knots in integral homology three-spheres. It is not hard to adapt their argument to prove the theorems for null-homologous knots in $3$--manifolds with torsion Spin$^c$ structures. See \cite[Section~2]{LR} for a detailed proof of the twisted formula. Our contribution here is to generalize the theorems to the case of non-torsion Spin$^c$ structures.

\begin{thm}\label{thm:ZeroMC}
Let $Y,K,C$ be as in Subsection~\ref{Subsect:KnotFloer}.
For any $k\in \mathbb Z$,
$HF^+(Y_0(K),\mathfrak t_{k})$ is isomorphic to the homology of the mapping cone of $$v^+_{k}+h^+_{k}\co A^+_{k}\to B^+.$$
\end{thm}

\begin{rem}
Using the same argument as in the proof of Theorem~\ref{thm:ZeroMC}, one should be able to extend the Dehn surgery formula in \cite{OSzIntSurg,OSzRatSurg} to null-homologous knots in any closed oriented $3$--manifolds.
\end{rem}

Ozsv\'ath and Szab\'o \cite{OSzAnn2} defined the universal twisted Heegaard Floer chain complex $\underline{CF^{\circ}}(Y)$ and the corresponding homology $\underline{HF^{\circ}}(Y)$ as modules over the group ring $\mathbb Z[H^1(Y)]$.
Suppose that $K\subset Y$ is a null-homologous knot. Let $[\mu]$ be the homology class of a meridian of $K$ in $Y_0(K)$, then the evaluation over $[\mu]$ defines a ring homomorphism $\mathbb Z[H^1(Y)]\to \mathbb Z[\mathbb Z]=\mathbb Z[T,T^{-1}]$. Thus we get the twisted Heegaard Floer chain complex $\underline{CF^{\circ}}(Y_0(K);\mathbb Z[T,T^{-1}])$ as a module over $\mathbb Z[T,T^{-1}]$. Its homology is denoted $\underline{HF^{\circ}}(Y_0(K);\mathbb Z[T,T^{-1}])$.

\begin{thm}\label{thm:ZeroMCTwisted}
Let $Y,K,C$ be as in Subsection~\ref{Subsect:KnotFloer}.
For any $k\in \mathbb Z$, the twisted Heegaard Floer homology
$\underline{HF^+}(Y_0(K),\mathfrak t_{k};\mathbb Z[T,T^{-1}])$ is isomorphic to the homology of the mapping cone of $$v^+_{k}+Th^+_{k}\co A^+_{k}[T,T^{-1}]\to B^+[T,T^{-1}].$$
\end{thm}

Given a chain map $f\co A\to B$, let $MC(f)$ be the mapping cone of $f$, namely, $MC(f)=A\oplus B$, with the differential 
$$\partial=\begin{pmatrix}
\partial_A &0\\
f &\partial_B
\end{pmatrix}.
$$

In \cite[Section~9]{OSzAnn2},
Ozsv\'ath and Szab\'o constructed a genus $h$ Heegaard quadruple
$$(\Sigma,\mbox{\boldmath${\alpha}$},\mbox{\boldmath${\beta}$},\mbox{\boldmath${\gamma}$}, \mbox{\boldmath${\delta}$}, w),$$
such that
\begin{itemize}
\item $(\Sigma,\mbox{\boldmath${\alpha}$},\mbox{\boldmath${\beta}$}-\{\beta_h\})$ is a Heegaard diagram for $Y\Bslash K$, and $\beta_h$ represents a meridian of $K$;
\item $\gamma_i,\delta_i$ are small Hamiltonian translates of $\beta_i$, $i=1,\dots,h-1$;
\item $\gamma_h$ represents the canonical longitude of $K$, and $\delta_h$ is isotopic to the juxtaposition of the $n$--fold juxtaposition of $\beta_h$ with $\gamma_h$;
\item all the necessary admissibility conditions are satisfied.
\end{itemize}

The three diagrams
$$
(\Sigma,\mbox{\boldmath${\alpha}$},\mbox{\boldmath${\beta}$}),(\Sigma,\mbox{\boldmath${\alpha}$},\mbox{\boldmath${\gamma}$}), (\Sigma,\mbox{\boldmath${\alpha}$},\mbox{\boldmath${\delta}$})
$$
represent $Y, Y_0(K), Y_n(K)$, respectively.

When $n$ is sufficiently large, $\mathfrak t_k$ is the only Spin$^c$ structure in 
$$\Big\{\mathfrak t_{k+ln}\in \spinc(Y_0(K))\Big|\;l\in\mathbb Z\Big\}$$ which is represented by an intersection point in $\mathbb T_{\alpha}\cap\mathbb T_{\gamma}$.

We will use the notations from Subsection~\ref{Subsect:KnotFloer}. Let $$CF^+(Y_0(K),[\mathfrak t_k])=\bigoplus_{l\in\mathbb Z}CF^+(Y_0(K),\mathfrak t_{k+ln}).$$
Then $$CF^+(Y_0(K),[\mathfrak t_k])\cong CF^+(Y_0(K),\mathfrak t_k)$$
by our choice of $n$.

In \cite[Section~9]{OSzAnn2},
Ozsv\'ath and Szab\'o proved that the sequence 
$$
\xymatrix{
0\ar[r] &CF^+(Y_0(K),[\mathfrak t_k])\ar[r]^-{f^+_2}&CF^+(Y_n(K),\mathfrak s_t)\ar[r]^-{f^+_3} &CF^+(Y,\mathfrak s)\ar[r]&0
}
$$
is exact at $CF^+(Y_0(K),[\mathfrak t_k])$ and $CF^+(Y,\mathfrak s)$, and there exists a $U$--equivariant null-homotopy $$H\co CF^+(Y_0(K),[\mathfrak t_k])\to CF^+(Y,\mathfrak s)$$
for $f^+_3\circ f^+_2$. Here $f^+_2, f^+_3$ are maps induced by two-handle cobordisms, and
$$H([\mathbf x,i])
=\sum_{\mathbf z\in \mathbb T_{\alpha}\cap \mathbb T_{\beta}}\sum_{\phi\in\pi_2(\mathbf x,\Theta_{\gamma,\delta},\Theta_{\delta,\beta},\mathbf z),\mu(\phi)=-1}\#\mathcal M(\phi)[\mathbf z,i-n_w(\phi)].$$

For any integer $\delta\ge0$, we will consider the subcomplex $CF^{\delta}$ of $CF^+$, generated by $[\mathbf x,i]$ with $0\le i\le \delta$. Similarly, let $A^{\delta}_k,B^{\delta}$ be the corresponding subcomplexes of $A^+_k,B^+$ which are kernels of $U^{\delta+1}$. 
Let $f^{\delta}_2, f^{\delta}_3, H^{\delta}$ be the restrictions of $f^+_2, f^+_3, H$ to the corresponding $CF^{\delta}$, and let $v^{\delta}_k,h^{\delta}_k\co A^{\delta}_k\to B^{\delta}$ be the restrictions of $v^+_k,h^+_k$.

Define
$$\psi^{\delta}\co CF^{\delta}(Y_0(K),\mathfrak t_{k})\to MC(f^{\delta}_3), \text{ and }\psi^{+}\co CF^+(Y_0(K),\mathfrak t_{k})\to MC(f^+_3)$$
by
$$\psi^{\delta}=(f^{\delta}_2,H^{\delta}), \text{ and }\psi^+=(f^+_2,H).$$

\begin{thm}\label{thm:OriginalMC}
The $U$--equivariant chain map 
\[
\psi^{+}\co CF^+(Y_0(K),\mathfrak t_{k})\to MC(f^+_3)
\]
is a quasi-isomorphism.
\end{thm}

Various versions of the above folklore theorem has been cited in the literature many times, but we are not aware of any complete proof of it. A sketch of a proof of Theorem~\ref{thm:OriginalMC} with $\mathbb F_2=\mathbb Z/2\mathbb Z$ coefficients will be given in the appendix.

\begin{proof}[Proof of Theorem~\ref{thm:ZeroMC} when $c_1(\mathfrak s)$ is torsion]
In this case, there exists an absolute $\mathbb Q$--grading on $CF^{+}(Y_n(K),\mathfrak s_t)$ and $CF^+(Y,\mathfrak s)$.
As in \cite[Lemma~4.4]{OSzIntSurg}, 
 there exists an $N>0$ such that whenever $n\ge N$ and $\mathfrak S\in \spinc(W'_n(K))$ induces a nontrivial map
$$f^{\delta}_{W'_n(K),\mathfrak S}\co CF^{\delta}(Y_n(K),\mathfrak s_t)\to CF^{\delta}(Y,\mathfrak s),$$
$\mathfrak S$ must be $\mathfrak x_k$ or $\mathfrak y_k$.
In particular, by Theorem~\ref{thm:LargeSurg}, the two chain complexes 
$MC(f^{\delta}_3)$ and $MC(v^{\delta}_{k}+h^{\delta}_{k})$
are equal.
It follows from Theorem~\ref{thm:OriginalMC} that
$\psi^{\delta}$ is a quasi-isomorphism.

If $k=0$, there exists a $\mathbb Z$--grading on $CF^+(Y_0(K),\mathfrak t_{0})$. Moreover, since $v^+_{0}$ and $h^+_{0}$ have the same grading shift, $MC(v^+_{0}+h^+_{0})$ is also $\mathbb Z$--graded. By \cite[Lemma~2.7]{OSzIntSurg},
the map
$$(\overline{\psi^+})_*\co HF^+(Y_0(K),\mathfrak t_{0})\to H_*(MC(v^+_{0}+h^+_{0}))$$
is a grading preserving $\mathbb Z[U]$--isomorphism.

If $k\ne0$, $c_1(\mathfrak t_{k})$ is nontorsion. There exists a $\delta\ge0$, such that the map $U^{\delta+1}$ on
$HF^+(Y_0(K),\mathfrak t_k)$ is zero. So the short exact sequence
$$
\xymatrix{
0\ar[r]&CF^{\delta}(Y_0(K),\mathfrak t_k)\ar[r]&CF^+(Y_0(K),\mathfrak t_k)\ar[r]^-{U^{\delta+1}}&CF^+(Y_0(K),\mathfrak t_k)\ar[r]&0
}
$$
gives rise to the short exact sequence 
$$
\xymatrix{
0\ar[r] &HF^+(Y_0(K),\mathfrak t_k)\ar[r] &HF^{\delta}(Y_0(K),\mathfrak t_k)\ar[r] &HF^+(Y_0(K),\mathfrak t_k)\ar[r] &0.
}
$$
Similarly, we have
$$
\xymatrix{
0\ar[r] &H_*(MC(v^+_{k}+h^+_{k}))\ar[r] &H_*(MC(v^{\delta}_{k}+h^{\delta}_{k}))\ar[r] &H_*(MC(v^+_{k}+h^+_{k}))\ar[r] &0.
}
$$
Hence we have a commutative diagram
\[
\xymatrix{
0\ar[r] &HF^+(Y_0(K),\mathfrak t_k)\ar[r]\ar[d]^{(\overline{\psi^+})_*} &HF^{\delta}(Y_0(K),\mathfrak t_k)\ar[r]\ar[d]^{(\psi^{\delta})_*} &HF^+(Y_0(K),\mathfrak t_k)\ar[r]\ar[d]^{(\overline{\psi^+})_*} &0\\
0\ar[r] &H_*(MC(v^+_{k}+h^+_{k}))\ar[r] &H_*(MC(v^{\delta}_{k}+h^{\delta}_{k}))\ar[r] &H_*(MC(v^+_{k}+h^+_{k}))\ar[r] &0.
}
\]
Using the fact that $\psi^{\delta}$ is a quasi-isomorphism, we see that \[(\overline{\psi^+})_*\co HF^+(Y_0(K),\mathfrak t_k)\to H_*(MC(v^+_{k}+h^+_{k}))\] is both injective and surjective.
\end{proof}

Given $k\in\mathbb Z$ and $n>0$, let
$$\mathcal X_k=\big\{ \mathfrak x\in\spinc(W'_n(K))\big|\: \langle c_1(\mathfrak x),\widehat{\varphi}_n\rangle\equiv 2k-n\pmod{2n},\quad \mathfrak x|Y=\mathfrak s\big\}.$$

For the case when $c_1(\mathfrak s)$ is nontorsion, we need the following proposition.

\begin{prop}\label{prop:FactorU}
Suppose that $n\ge \max\{2|k|+g(F),2g(F)\}$, $\mathfrak x\in\mathcal X_k\setminus\{\mathfrak x_k,\mathfrak y_k\}$, then the map
$F^{-}_{W'_n(K),\mathfrak x}$ factorizes through the map
\[
U^{n-2|k|-g}\co HF^{-}(Y_n(K),\mathfrak s_t)\to HF^{-}(Y_n(K),\mathfrak s_t).
\]
Moreover, if $c_1(\mathfrak s)$ is nontorsion, then $F^{+}_{W'_n(K),\mathfrak x}$ factorizes through
$$
U^{n-2|k|-g}\co HF^{+}(Y_n(K),\mathfrak s_t)\to HF^{+}(Y_n(K),\mathfrak s_t).
$$
\end{prop}

Our method of proving this proposition is taken from \cite[Theorem~3.1]{OSzSympl}.

Let $CF^{\le0}(Z,\mathfrak u)$ be the subcomplex of $CF^{\infty}(Z,\mathfrak u)$ which consists of $[\mathbf x,i]$, $i\le0$. This chain complex is clearly isomorphic to
$CF^{-}(Z,\mathfrak u)$ via the $U$--action.
Ozsv\'ath and Szab\'o \cite{OSzAbGr} defined a natural transformation
$$F^{\circ}_{Y\#Z,\mathfrak t\#\mathfrak u}\co HF^{\circ}(Y,\mathfrak t)\otimes HF^{\le0}(Z,\mathfrak u)\to HF^{\circ}(Y\#Z,\mathfrak t\#\mathfrak u).$$

Suppose that $W\co Y_1\to Y_2$ is a cobordism equipped with a Spin$^c$ structure $\mathfrak S\in\spinc(W)$, $\mathfrak S|Y_i=\mathfrak t_i$.
Then there is a commutative diagram \cite[Proposition~4.4]{OSzAbGr}
\begin{equation}\label{eq:le0Comm}
\xymatrixcolsep{5pc}
\xymatrix{ HF^{\circ}(Y_1,\mathfrak t_1)\otimes HF^{\le0}(Z,\mathfrak u)\ar[r]^-{F^{\circ}_{Y_1\#Z,\mathfrak t_1\#\mathfrak u}}
\ar[d]^-{F^{\circ}_{W,\mathfrak S}\otimes\mathrm{id}}& HF^{\circ}(Y_1\#Z,\mathfrak t_1\#\mathfrak u)\ar[d]^-{F^{\circ}_{W\#(Z\times[0,1]),\mathfrak S\#\mathfrak u}}\\
 HF^{\circ}(Y_2,\mathfrak t_2)\otimes HF^{\le0}(Z,\mathfrak u) \ar[r]^-{F^{\circ}_{Y_2\#Z,\mathfrak t_2\#\mathfrak u}} & HF^{\circ}(Y_2\#Z,\mathfrak t_2\#\mathfrak u)
}
\end{equation}

\begin{proof}[Proof of Proposition~\ref{prop:FactorU}]
As in the proof of \cite[Theorem~3.1]{OSzSympl}, we will first do a model computation. Let $N=\nu(\widehat F_n)\Bslash B^4$ be a punctured neighborhood of $\widehat F_n\subset W'_n(K)$. Let $V=\#^{2g}(S^1\times S^2)$, then
$N$ is a cobordism from $S^3$ to $V_{-n}(B_g)$, where $B_g\subset V$ is the genus $g$ Borromean knot. 

Recall from \cite[Section~9]{OSzKnot} that $$HFK^{\infty}(V,B_g)\cong \mathbb Z[U,U^{-1}]\otimes\Lambda^*(\mathbb Z^{2g})$$ is generated by $\mathbf 1\otimes\theta$ as a module over $\mathbb Z[U,U^{-1}]\otimes\Lambda^*(\mathbb Z^{2g})$. Here $$\theta\in \widehat{HF}(V)\cong \Lambda^*(\mathbb Z^{2g})$$ is the generator with the highest Maslov grading. The double filtration level of $\mathbf 1\otimes\theta$ is $(0,g)$, and the double filtration level of $HFK^{\infty}(V,B_g)$ is supported in the set $$\Big\{(i,j)\in\mathbb Z^2\Big|\:|i-j|\le g\Big\}.$$

The map 
$F^{\le0}_{N,\mathfrak x|N}$
is the composition of two maps: The first map sends $HF^{\le0}(S^3)$ isomorphically to $\mathbb Z[U]\otimes\theta\subset HF^{\le0}(V)$, the second map is induced by the two-handle cobordism
$V\to V_{-n}(B_g)$. Since we assume $n\ge2g(F)$, by \cite[Remark~4.3]{OSzKnot}, we can apply \cite[Theorem~4.1, Corollary~4.2]{OSzKnot}. In particular,
when $\mathfrak x=\mathfrak x_k$,  the second map is represented by a map
$$\prescript{b}{}v^{\le0}_k\co CFK^{\infty}(V,B_g)\{i\le0\}\to CFK^{\infty}(V,B_g)\{i\le0\text{ or }j\le -k\},$$
and $(\prescript{b}{}v^{\le0}_k)_*$ maps $\mathbb Z[U]\otimes\theta$ isomorphically to $\mathbb Z[U]\otimes\theta\subset HFK^{\infty}(V,B_g)$.  

Suppose that $\mathfrak x\in \mathcal X_k\setminus\{\mathfrak x_k,\mathfrak y_k\}$, then $\mathfrak x$ is equal to $\mathfrak x_k+l\mathrm{PD}[\widehat F_n]$ for $l\ne0,-1$.
The grading difference of $F^{\le0}_{N,\mathfrak x_k|N}(\mathbf 1)$ and $F^{\le0}_{N,(\mathfrak x_k+l\mathrm{PD}[\widehat F_n])|N}(\mathbf 1)$
is
\begin{eqnarray*}
&&\frac{c^2_1(\mathfrak x_k|N)-c_1^2((\mathfrak x_k+l\mathrm{PD}[\widehat F_n])|N)}4\\
&=&\frac{-(2k-n)^2+(2k-n-2ln)^2}{4n}=l(ln+n-2k).
\end{eqnarray*}
So 
$F^{\le0}_{N,\mathfrak x|N}(\mathbf 1)$ is contained in the same Maslov grading level as
$$U^{\frac{l(l+1)n}2-lk}\otimes\theta,$$
whose double filtration level is $$(-\frac{l(l+1)n}2+lk,-\frac{l(l+1)n}2+lk+g).$$
In $HFK^{\infty}(V,B_g)$, the Maslov grading is equal to $i+j$ up to an overall translation, hence the double filtration level of $F^{\le0}_{N,\mathfrak x|N}(\mathbf 1)$
is contained in the range
$$(-\frac{l(l+1)n}2+lk+m,-\frac{l(l+1)n}2+lk+g-m),\quad m=0,1,\dots,g.$$
Any element in $HFK^{\infty}(V,B_g)$ with the above filtration levels can be obtained from $$U^{\frac{l(l+1)n}2-lk-g}\otimes\theta,$$
whose double filtration level is $$(-\frac{l(l+1)n}2+lk+g,-\frac{l(l+1)n}2+lk+2g),$$ by applying the action of an element in  $\mathbb Z[U]\otimes\Lambda^*(H_1(V))$.
Since $l\ne0,-1$ and $n\ge2|k|$, we have $$\frac{l(l+1)n}2-lk-g\ge n-2|k|-g.$$ So the map $F^{\le0}_{N,\mathfrak x|N}$ factorizes through $$U^{n-2|k|-g}\co HF^{\le0}(S^3)\to HF^{\le0}(S^3).$$

As in the proof of \cite[Theorem~3.1]{OSzSympl}, we can decompose the cobordism $W'_n(K)$ as the composition of two cobordisms
$$W_1\co Y_n(K)\to V_{-n}(B_g)\#Y_n(K)$$
and $$W_2\co V_{-n}(B_g)\#Y_n(K)\to Y,$$
where $W_1$ is a ``timewise'' connected sum of $N$ and $Y_n(K)\times[0,1]$. By (\ref{eq:le0Comm}), we have
\[
\xymatrixcolsep{5pc}
\xymatrix{ HF^{-}(S^3)\otimes HF^{\le0}(Y_n(K),\mathfrak s_t)\ar[r]^-{F^{-}_{S^3\#Y_n(K),\mathfrak s_t}}
\ar[d]^-{F^{-}_{N,\mathfrak x|N}\otimes\mathrm{id}}& HF^{-}(Y_n(K),\mathfrak s_t)\ar[d]^-{F^{-}_{W_1,\mathfrak x|N\#\mathfrak s_t}}\\
 HF^{-}(V_{-n}(B_g))\otimes HF^{\le0}(Y_n(K),\mathfrak s_t) \ar[r]^-{F^{-}_{V_{-n}(B_g)\#Y_n(K),\mathfrak s_t}} & HF^{-}(V_{-n}(B_g)\#Y_n(K),\mathfrak s_t).
}
\]
The horizontal map in the first row
is surjective. Since $F^{\le0}_{N,\mathfrak x|N}$ and thus $F^{-}_{N,\mathfrak x|N}$ factorize through $U^{n-2|k|-g}$, $F^{-}_{W_1,\mathfrak x|W_1}$ also factorizes through $U^{n-2|k|-g}$.
By Lemma~\ref{lem:SpincInj}, $F^{-}_{W'_n(K),\mathfrak x}$ factorizes through $F^{-}_{W_1,\mathfrak x|W_1}$, so $F^{-}_{W'_n(K),\mathfrak x}$ factorizes through $U^{n-2|k|-g}$.

When $c_1(\mathfrak s)$ is nontorsion, using the natural short exact sequence (\ref{eq:SES+-inf}),
we see that $F^{+}_{W'_n(K),\mathfrak x}$ also factorizes through
$U^{n-2|k|-g}$, since $HF^+$ is just a submodule of $HF^-$.
\end{proof}

\begin{proof}[Proof of Theorem~\ref{thm:ZeroMC} when $c_1(\mathfrak s)$ is nontorsion]
As before, Theorem~\ref{thm:OriginalMC} implies that
$CF^+(Y_0(K),\mathfrak t_k)$ is quasi-isomorphic to
$MC(f^+_{3})$,
where $$f^+_3=\sum_{\mathfrak x\in\mathcal X_k}f^+_{W'_n(K),\mathfrak x}\co CF^+(Y_n(K),\mathfrak s_t)\to CF^+(Y,\mathfrak s).$$
Since $c_1(\mathfrak s)$ is nontorsion, $H_*(A^+_k)$ is a finitely generated abelian group. So $U^m|H_*(A^+_k)=0$ when $m$ is greater than a constant $C_1$ independent of $n$. 
By Theorem~\ref{thm:LargeSurg}, this implies that $U^m|HF^+(Y_n(K),\mathfrak s_t)=0$ when $m>C_1$.
Proposition~\ref{prop:FactorU} implies that $F^+_{W'_n(K),\mathfrak x}=0$ when  $\mathfrak x\in\mathcal X_k\setminus\{\mathfrak x_k,\mathfrak y_k\}$
and $n>\max\{2|k|+g,2g\}+C_1$. Our conclusion follows by Theorem~\ref{thm:LargeSurg}.
\end{proof}

\begin{proof}[Proof of Theorem~\ref{thm:ZeroMCTwisted}]
We use the fact that 
$\underline{CF}^+(Y_0(K),\mathfrak t_k;\mathbb Z[T,T^{-1}])$ is quasi-isomorphic to
$MC(\underline{f^+_3})$,
where $$\underline{f^+_3}=\sum_{\mathfrak x\in\mathcal X_k}T^{\frac1n\langle \mathfrak x-\mathfrak x_k,\widehat{\varphi}_n\rangle}f^+_{W'_n(K),\mathfrak x}\co CF^+(Y_n(K),\mathfrak s_t)[T,T^{-1}]\to CF^+(Y,\mathfrak s)[T,T^{-1}].$$
The corresponding long exact sequence can be found in \cite[Theorem~9.23]{OSzAnn2} and the last formula in \cite[Section~3]{OSzGenus}. The fact that the two chain complexes are quasi-isomorphic can be proved in a similar way as Theorem~\ref{thm:OriginalMC}.
The rest of the argument is the same as in the proof of Theorem~\ref{thm:ZeroMC}. 
\end{proof}

\section{The zero surgery}\label{sect:ZeroSurg}

In this section, we will prove our main results.
For simplicity, we will use coefficients in a fixed field $\mathbb F$ for Heegaard Floer homology throughout this section. We may choose $\mathbb F=\mathbb F_2$ since we only prove Theorem~\ref{thm:OriginalMC} with $\mathbb F_2$ coefficients in the Appendix, but other fields can also be used if we assume Theorem~\ref{thm:OriginalMC} with $\mathbb Z$ coefficients.

\subsection{A rank inequality}

\begin{lem}\label{lem:vhIm}
Suppose that $Y,K,F,X,G$ satisfy the same condition as in Theorem~\ref{thm:4genus}, then
\[\mathrm{im}(v_{k}^+)_*\supset \mathrm{im}(h_{k}^+)_*,\]
when $k\ge g(G)$. Moreover, if $G$ is a disk, then $$(v^+_0)_*=(h^+_0)_*.$$
\end{lem}
\begin{proof}
Without loss of generality, we may assume $k=g(G)$, since we can always increase the genus of $G$.
Let $W=W'_n(K)\#N\overline{\mathbb CP^2}$, $\Sigma\subset W$ be the closed oriented surface obtained by capping off $-G$ with the cocore of the $2$--handle in $W'_n(K)$, and $\mathfrak S$ be the connected sum of $\mathfrak y_k$ with $N$ copies of $\mathfrak t_1$, where $\mathfrak t_1\in\spinc(\overline{\mathbb CP^2})$ satisfies that $c_1(\mathfrak t_1)$ generates $H^2(\overline{\mathbb CP^2})$. It follows from (\ref{eq:xkyk}) that
\[
\langle c_1(\mathfrak y_k\#N\mathfrak t_1),[\Sigma]\rangle=-2k-n,\quad \mathfrak x_k\#N\mathfrak t_1=\mathfrak y_k\#N\mathfrak t_1-\mathrm{PD}[\Sigma].
\]
Our conclusion follows from Theorems~\ref{thm:LargeSurg} and~\ref{thm:AdjRel} and the blow-up formula \cite[Theorem~1.4]{OSzFour}.
\end{proof}

The following proposition is an analogue of \cite[Corollary~4.5]{OSzKnot}.

\begin{prop}\label{prop:ZeroTorsEqual}
Let $K$ be a null-homologous knot in $Y$, $\varphi$ be a Seifert homology class, $\mathfrak s\in \spinc(Y)$. 
Let $$d=\max\left\{i\in\mathbb Z\left| \widehat{HFK}(Y,K,\mathfrak s,\varphi,i)\ne0\right.\right\}.$$
If 
\begin{equation}\label{eq:vhIm2}
\mathrm{im}(v_{d-1})_*\supset \mathrm{im}(h_{d-1})_*,
\end{equation}
and
one of the following conditions holds:
\begin{itemize}
\item $c_1(\mathfrak s)$ is nontorsion and $d\ge1$,
\item $HF^+(Y,\mathfrak s)=0$,
\item $d>1$,
\end{itemize}
then
$$\mathrm{rank\:}HF^+(Y_0(K),\mathfrak t_{d-1})\ge\mathrm{rank\:}\widehat{HFK}(Y,K,\mathfrak s,\varphi,d).$$
\end{prop}
\begin{proof}
As in the proof of \cite[Corollary~4.5]{OSzKnot}, there is an exact triangle
\begin{equation}\label{eq:v(d-1)}
\xymatrix{
H_*(A^+_{d-1})\ar[r]^-{(v^+_{d-1})_*}&H_*(B^+)\ar[ld]\\ 
H_*(C\{(-1,d-1)\})\ar[u]&  
}.
\end{equation}
For simplicity, denote $(v^+_{d-1})_*=v$, $(h^+_{d-1})_*=h$. Then $\mathrm{im\:}h\subset \mathrm{im\:}v$ by (\ref{eq:vhIm2}).
It follows from (\ref{eq:v(d-1)}) that
\begin{equation}\label{eq:RankHFK}
\mathrm{rank\:}\widehat{HFK}(Y,K,\mathfrak s,\varphi,d)=\mathrm{rank\:}\ker v+\mathrm{rank\:}\mathrm{coker\:}v.
\end{equation}
By Theorem~\ref{thm:ZeroMC},
\begin{equation}\label{eq:RankHF+}
\mathrm{rank\:}HF^+(Y_0(K),\mathfrak t_{d-1})=\mathrm{rank\:}\ker(v+h)+\mathrm{rank\:}\mathrm{coker}(v+h).
\end{equation}

\noindent{\bf Case 1.}
 $c_1(\mathfrak s)$ is nontorsion and $d\ge1$.

 In this case both $H_*(A^+_{d-1})$ and $H_*(B^+)$ have finite ranks. It follows from (\ref{eq:RankHFK}) and (\ref{eq:RankHF+}) that
$$\mathrm{rank\:}\widehat{HFK}(Y,K,\mathfrak s,\varphi,d)=\mathrm{rank\:}H_*(A^+_{d-1})+\mathrm{rank\:}H_*(B^+)-2\mathrm{rank\:}\mathrm{im\:} v,$$
$$\mathrm{rank\:}HF^+(Y_0(K),\mathfrak t_{d-1})=\mathrm{rank\:}H_*(A^+_{d-1})+\mathrm{rank\:}H_*(B^+)-2\mathrm{rank\:}\mathrm{im}(v+h).$$
Our conclusion follows from (\ref{eq:vhIm2}).

\vspace{5pt}\noindent{\bf Case 2.} $HF^+(Y,\mathfrak s)=0$.

Since $H_*(B^+)=0$, the right hand sides of both (\ref{eq:RankHFK}) and (\ref{eq:RankHF+}) are $\mathrm{rank\:}H_*(A_d^+)$.

\vspace{5pt}\noindent{\bf Case 3.}
 $d>1$.

In this case we may assume $c_1(\mathfrak s)$ is torsion, then there is an absolute $\mathbb Q$--grading on $C$.
Let $\rho\co \mathrm{im\:}v\to H_*(A^+_{d-1})$ be a homogeneous homomorphism of $\mathbb F$--vector
spaces such that $$v\circ\rho=\mathrm{id}|\mathrm{im\:}v.$$
By (\ref{eq:vhIm2}), the map $\mathrm{id}+\rho h\co H_*(A^+_{d-1})\to H_*(A^+_{d-1})$ is well-defined.
Clearly, $\mathrm{id}+\rho h$ maps $\ker(v+h)$ to $\ker v$.

Since $d>1$, the grading shift of $h$ is strictly less than the grading shift of $v$, so the grading shift of $\rho h$ is negative. As the grading of $H_*(A^+_{d-1})$ is bounded from below, for any $x\in H_*(A^+_{d-1})$, $(\rho h)^m(x)=0$ when $m$ is sufficiently large. The map
$$\mathrm{id}-\rho h+(\rho h)^2-(\rho h)^3+\cdots\co H_*(A^+_{d-1})\to H_*(A^+_{d-1})$$
is well-defined, and it maps $\ker v$ to $\ker (v+h)$. This map is the inverse to $\mathrm{id}+\rho h\co\ker(v+h)\to \ker v$,
so $$\mathrm{rank}\ker(v+h)=\mathrm{rank}\ker v.$$
Since $\mathrm{im\:}(v+h)\subset \mathrm{im\:}v$,  $$\mathrm{rank\:}\mathrm{coker}(v+h)\ge\mathrm{rank\:}\mathrm{coker\:} v.$$
Our conclusion holds by (\ref{eq:RankHFK}), (\ref{eq:RankHF+}).
\end{proof}

\subsection{The proof of Theorem~\ref{thm:4genus} when $g(F)>1$}\label{subsect:4genus}

In this subsection, we will prove Theorem~\ref{thm:4genus} in the case $g(F)>1$. If $g(F)=1$, $G$ has to be a disk, and this case will be treated in Subsection~\ref{subsect:Genus1}.

\begin{proof}[Proof of Theorem~\ref{thm:4genus} when $g(F)>1$]
We first prove Property G1. Since $F$ is taut, we get $\widehat{HFK}(Y,K,[F],g(F))\ne0$ by Theorem~\ref{thm:SeifertNorm}. Using Lemma~\ref{lem:vhIm} and Proposition~\ref{prop:ZeroTorsEqual}, we see that $HF^+(Y_0(K),[\widehat F],g(F)-1)\ne0$, so $\widehat F$ is also taut.

Next, we prove Property G2. Suppose that $Y_0(K)$ is a surface bundle over $S^1$, with $S$ being the fiber in the homology class $\widehat{\varphi}$. By Property G1, $g(S)>0$, so $Y_0(K)$ is irreducible. If $Y\Bslash K$ is reducible, let $P\subset Y\Bslash K$ be an essential sphere. Since $Y_0(K)$ is irreducible, $P$ must bound a ball in $Y_0(K)$, and this ball must contain the dual knot of $K$. This is not possible because the dual knot of $K$ is not null-homologous in $Y_0(K)$. So $Y\Bslash K$ is irreducible.

Since $Y_0(K)$ fibers over $S^1$ with fiber in the homology class $[\widehat F]$, $\widehat F$ is isotopic to a fiber. (A reference for this folklore result can be found in the last sentence of the first paragraph in \cite[Section~3]{Thurston}.) 
We have $HF^+(Y_0(K),[\widehat F],g(F)-1)\cong\mathbb F$ by \cite[Theorem~5.2]{OSzSympl}. Using Lemma~\ref{lem:vhIm} and Proposition~\ref{prop:ZeroTorsEqual}, we see that \[\mathrm{rank\:}\widehat{HFK}(Y,K,[F],g(F))\le1.\] However, we already know $\widehat{HFK}(Y,K,[F],g(F))\ne0$, so $\widehat{HFK}(Y,K,[F],g(F))\cong\mathbb F$, and our conclusion follows from Theorem~\ref{thm:KnotFibered}.
\end{proof}

\subsection{The proof of Theorem~\ref{thm:Slice}}

Let $Y,K,X$ be as in Theorem~\ref{thm:Slice}, and let $G\subset X$ be the disk bounded by $K$. Without loss of generality, we will assume $K$ is not the unknot. There is a natural map 
\[
\pi\co X\to Y
\]
which is the composition of the pinching map $X\to Y\times[0,1]$ with the projection $Y\times[0,1]\to Y$.
Suppose that $\varphi_0=[\pi(G)]\in H_2(Y,K)$ is the homology class of the immersed disk $\pi(G)$, and $\varphi\in H_2(Y,K)$ is a Seifert homology class for $K$. Let $\widehat{\varphi}_0,\widehat{\varphi}\in H_2(Y_0(K))$ be the extensions of $\varphi_0,\varphi$.
Given a Spin$^c$ structure $\mathfrak s\in\spinc(Y)$,
let $\mathfrak t_{k}\in\spinc(Y_0(K))$ be the extension of $\mathfrak s|(Y\Bslash K)$ with $$\langle c_1(\mathfrak t_{k}),\widehat{\varphi}_0\rangle=2k,$$
and let $\xi_k\in\relspinc(Y,K)$ be the relative Spin$^c$ structure whose underlying Spin$^c$ structure is $\mathfrak s$, and $$\langle c_1(\xi_k),\varphi_0\rangle=2k+1.$$
Let $C=CFK^{\infty}(Y,K,\xi_0,\varphi_0)$.

\begin{lem}\label{lem:SpincDiff}
With the above notation, we have $$\langle c_1(\xi_k),\varphi\rangle-\langle c_1(\xi_l),\varphi\rangle=\langle c_1(\xi_k),\varphi_0\rangle-\langle c_1(\xi_l),\varphi_0\rangle=2k-2l.$$
\end{lem}
\begin{proof}
Since $\varphi-\varphi_0$ is represented by a closed surface in $Y\Bslash K$, we have
\begin{eqnarray*}
\langle c_1(\xi_k),\varphi\rangle-2k-1&=&\langle c_1(\xi_k),\varphi-\varphi_0\rangle\\
&=&\langle c_1(\mathfrak s),\varphi-\varphi_0\rangle
\end{eqnarray*}
is independent of $k$.
\end{proof}

\begin{lem}\label{lem:Top>0}
Let $F$ be a taut Seifert-like surface for a nontrivial knot $K$, and let
$$\Xi_{\max}=\left\{\xi\in \relspinc(Y,K)\left|\langle c_1(\xi),[F]\rangle=2+\chi_-(F)\right. \right\}.$$
Suppose $\mathfrak s\in\spinc(Y)$ is the underlying Spin$^c$ structure of some $\xi\in\Xi_{\max}$.
If $HF^+(Y,\mathfrak s)\ne0$, then $\langle c_1(\xi),\varphi_0\rangle\ge3$.
\end{lem}
\begin{proof}
Gluing $\pi(-G)$ to $F$, we get an immersed closed surface with $\chi_-=\chi_-(F)-1$ in $Y$ representing the homology class $\varphi-\varphi_0$. By Gabai's theorem that the singular Thurston norm is equal to the Thurston norm \cite[Corollary~6.18]{G1}, the Thurston norm of $\varphi-\varphi_0\in H_2(Y)$ is at most $\chi_-(F)-1$.
Since $HF^+(Y,\mathfrak s)\ne0$, it follows from the adjunction inequality that
\begin{eqnarray*}
\chi_-(F)-1&\ge&\langle c_1(\mathfrak s),\varphi-\varphi_0\rangle\\
&=&\langle c_1(\xi),\varphi-\varphi_0\rangle\\
&=&\chi_-(F)+2-\langle c_1(\xi),\varphi_0\rangle.
\end{eqnarray*}
Hence $\langle c_1(\xi),\varphi_0\rangle\ge3$.
\end{proof}

\begin{proof}[Proof of Theorem~\ref{thm:Slice} for Property G1]
Suppose that $F$ is a taut Seifert-like surface for $K$.
Let $\Xi_{\max}$ be as in Lemma~\ref{lem:Top>0}, then there exists $\xi\in \Xi_{\max}$ with $\widehat{HFK}(Y,K,\xi)\ne0$ by Theorem~\ref{thm:SeifertNorm}.
Let $\mathfrak s\in\spinc(Y)$ be the underlying Spin$^c$ structure of $\xi$. By Lemma~\ref{lem:SpincDiff}, if $\mathfrak s$ is the underlying Spin$^c$ structure of another $\eta\in\relspinc(Y,K)$ with $\widehat{HFK}(Y,K,\eta)\ne0$, then 
\[\langle c_1(\xi)-c_1(\eta),\varphi_0\rangle=\langle c_1(\xi)-c_1(\eta),\varphi\rangle\ge0.\] 
Suppose that $\langle c_1(\xi),\varphi_0\rangle=2d+1$, then the above inequality means that
\begin{equation}\label{eq:dMax}
d=\max\left\{i\in\mathbb Z\left| \widehat{HFK}(Y,K,\mathfrak s,\varphi_0,i)\ne0\right.\right\}.
\end{equation}

If $\mathfrak s,d$ satisfy one of the three conditions in Proposition~\ref{prop:ZeroTorsEqual}, it follows from Proposition~\ref{prop:ZeroTorsEqual} and (\ref{eq:dMax}) that $HF^+(Y_0(K),\mathfrak t_{d-1})\ne0$.
If $\mathfrak s,d$ do not satisfy any of the three conditions in Proposition~\ref{prop:ZeroTorsEqual}, then $c_1(\mathfrak s)$ is torsion, and $d=1$ by Lemma~\ref{lem:Top>0}. Hence $\mathfrak t_{d-1}$ is also torsion, so we also have $HF^+(Y_0(K),\mathfrak t_{d-1})\ne0$.

We have
\begin{eqnarray}
\langle c_1(\mathfrak t_{d-1}),[\widehat F]\rangle&=&\langle c_1(\mathfrak t_{d-1}),[\widehat F]-\widehat{\varphi}_0\rangle+\langle c_1(\mathfrak t_{d-1}),\widehat{\varphi}_0\rangle\nonumber\\\nonumber
&=& \langle c_1(\xi),[F]-\varphi_0\rangle+2d-2\\\nonumber
&=&\chi_-(F)+2-(2d+1)+2d-2\\\nonumber
&=&\chi_-(F)-1\\
&=&\chi_-(\widehat F).\label{eq:chiF}
\end{eqnarray}
Since $HF^+(Y_0(K),\mathfrak t_{d-1})\ne0$, it follows from Theorem~\ref{thm:ThurstonNorm} that $\widehat F$ is Thurston norm minimizing.

In order to show that $\widehat F$ is taut, we only need to prove that if $T\subset \widehat F$ is a torus component, then $[T]$ is not represented by a sphere. If $F$ is disconnected, then $b_1(Y)>0$, and this case follows from Proposition~\ref{prop:b1>0}, which will be proved in Subsection~\ref{subsect:Genus1}. If $F$ is connected, then $\chi_-(F)=1$, and this case will also be proved in Subsection~\ref{subsect:Genus1}.
\end{proof}

\begin{proof}[Proof of Theorem~\ref{thm:Slice} for Property G2]
Without loss of generality, we may assume $K$ is not the unknot. As in the proof in Subsection~\ref{subsect:4genus}, $Y\Bslash K$ is irreducible.

If $g(S)>1$, by \cite[Theorem~5.2]{OSzSympl}, there exists
 $\mathfrak t\in \spinc(Y_0(K))$ which is the unique Spin$^c$ structure such that $HF^+(Y_0(K),\mathfrak t)\ne0$ and $\langle c_1(\mathfrak t),\widehat{\varphi}\rangle=2g(S)-2$. Moreover,
$HF^+(Y_0(K),\mathfrak t)\cong\mathbb F$.

Suppose $\xi\in \relspinc(Y,K)$ satisfies that $\langle c_1(\xi),\varphi\rangle=2g(S)+1$.
Let $\mathfrak s\in\spinc(Y)$ be the underlying Spin$^c$ structure of $\xi$, and let $\mathfrak r\in\spinc(Y_0(K))$ be the extension of $\xi$.

\noindent{\bf Claim.} If $\widehat{HFK}(Y,K,\xi)\ne0$, then $\mathfrak t=\mathfrak r-\mathrm{PD}[\mu]$, and $\widehat{HFK}(Y,K,\xi)\cong\mathbb F$.

If $\widehat{HFK}(Y,K,\xi)\ne0$,
let $d\in\mathbb Z$ be defined by
$$\langle c_1(\xi),\varphi_0\rangle=2d+1.$$
By Property G1, there exists no $\eta\in\relspinc(Y,K)$ such that $\langle c_1(\eta),\varphi\rangle>2g(S)+1$.
It follows from Lemma~\ref{lem:SpincDiff} that
(\ref{eq:dMax}) holds.

If $\mathfrak s,d$ satisfy one of the three conditions in Proposition~\ref{prop:ZeroTorsEqual}, then $$HF^+(Y_0(K),\mathfrak t_{d-1})\ne0.$$ As we have computed in (\ref{eq:chiF}),
$\langle c_1(\mathfrak t_{d-1}),\widehat{\varphi}\rangle=2g(S)-2$, hence we must have $\mathfrak t_{d-1}=\mathfrak t$.
It follows from Proposition~\ref{prop:ZeroTorsEqual} that $\widehat{HFK}(Y,K,\xi)\cong\mathbb F$.
Our claim follows.

If $\mathfrak s,d$ do not satisfy any of the three conditions in Proposition~\ref{prop:ZeroTorsEqual},
Lemma~\ref{lem:Top>0} implies that $d=1$ and $c_1(\mathfrak s)$ is torsion, we have
\begin{eqnarray*}
2g(S)+1&=&\langle c_1(\xi),\varphi\rangle\\
&=&\langle c_1(\xi),\varphi-\varphi_0\rangle+\langle c_1(\xi),\varphi_0\rangle\\
&=&\langle c_1(\mathfrak s),\varphi-\varphi_0\rangle+2d+1\\
&=&3,
\end{eqnarray*}
a contradiction to our assumption that $g(S)>1$.
This proves our claim.

Let $F$ be a taut Seifert-like surface for $K$ with $[F]=\varphi$. By Property G1, $\widehat F$ is taut in $Y_0(K)$. Since $[\widehat{F}]=\widehat{\varphi}$ is the homology class of the fiber, $\widehat{F}$ must be isotopic to the fiber. (See the proof in Subsection~\ref{subsect:4genus} for a reference for this result.) In particular, $F$ is connected. 
Our theorem in this case follows from the claim and Theorem~\ref{thm:KnotFibered}.

The case $g(S)=1$ will be treated in Subsection~\ref{subsect:Genus1}.
\end{proof}

\subsection{The $\chi_-(F)=1$ case}\label{subsect:Genus1}

In this subsection, we will prove the $\chi_-(F)=1$ case of Theorem~\ref{thm:Slice}, which also implies the $g(F)=1$ case of Theorem~\ref{thm:4genus}. We first give a proof of Proposition~\ref{prop:b1>0}.

\begin{proof}[Proof of Proposition~\ref{prop:b1>0}]
Let $M=Y\Bslash K$. Since $b_1(Y)>0$, there exists a closed, oriented, connected surface $S$ in the interior of $M$, such that $S$ is taut in $M$. Notice that for the $\infty$ slope on $K$, the core of the surgery solid torus in $Y_{\infty}(K)=Y$, which is $K$, is null-homotopic. Using a theorem of Lackenby \cite[Theorem~A.21]{Lackenby}, which is a stronger version of the main result in \cite{G2}, we conclude that each $2$--sphere in $Y_0(K)$ bounds a rational homology ball. Hence $Y_0(K)$ does not have an $S^1\times S^2$ connected summand.
\end{proof}

To prove the $\chi_-(F)=1$ case of Theorem~\ref{thm:Slice},
we use the argument in \cite{AiNi}. We will use twisted coefficients in the Novikov ring $\Lambda=\mathbb F[[T,T^{-1}]$.

\begin{prop}\label{prop:NoSphere1}
Let $Y,K,X$ be as in Theorem~\ref{thm:Slice}, and let $G\subset X$ be the disk bounded by $K$. Suppose that $K$ is nontrivial and $F$ is a genus--$1$ Seifert surface for $K$, then there does not exist an embedded sphere $P\subset Y_0(K)$ such that $[P]\cdot[\mu]\ne0$, where $\mu$ is the meridian of $K$.
\end{prop}
\begin{proof}
We will consider 
\[\underline{HF^{+}}(Y_0(K);\Lambda):=\underline{HF^{+}}(Y_0(K);\mathbb F[T,T^{-1}])\otimes_{\mathbb F[T,T^{-1}]} \Lambda.\]
As in \cite{NiNSSphere}, if there exists such a sphere $P$, then $\underline{HF^{+}}(Y_0(K);\Lambda)=0$.

By Lemma~\ref{lem:vhIm}, we have $(v^+_0)_*=(h^+_0)_*$. It follows from Theorem~\ref{thm:ZeroMCTwisted} that $(v^+_0)_*$ is an isomorphism, which would imply that $\widehat{HFK}(Y,K,[F],1)=0$ by the exact triangle (\ref{eq:RankHFK}), where $d=1$.
This contradicts Theorem~\ref{thm:SeifertNorm}.
\end{proof}

\begin{proof}[Proof of Theorem~\ref{thm:Slice} when $\chi_-(F)=1$]
We first prove Property G1.
Without loss of generality, we assume that $Y\Bslash K$ is irreducible.

Since $\chi_-(F)=1$ and $Y\Bslash K$ is irreducible, every component of $\widehat F$ is a torus. To prove Property G1, we just need to show that $Y_0(K)$ contains no non-separating $2$--spheres. This follows from 
Proposition~\ref{prop:NoSphere1} when $b_1(Y)=0$, and from
Proposition~\ref{prop:b1>0} when $b_1(Y)>0$. 

Now we prove Property G2. Suppose that $Y_0(K)$ is a torus bundle over $S^1$, with fiber in the homology class $\widehat{\varphi}$. As argued in the last subsection, $\varphi$ is represented by a genus--$1$ Seifert surface $F$, since we already know Property G1 in this case. By \cite{AiPeters}, $\underline{HF^{+}}(Y_0(K);\Lambda)\cong\Lambda$. By Lemma~\ref{lem:vhIm} and Theorem~\ref{thm:ZeroMCTwisted}, we see that the homology of the mapping cone of \[(1+T)(v_0^+)\co A^+_0[[T,T^{-1}]\to B^+[[T,T^{-1}]\] is $\Lambda$. Hence the homology of the mapping cone of \[v_0^+\co A^+_0\to B^+\] is $\mathbb F$.
Using (\ref{eq:v(d-1)}), we get $\widehat{HFK}(Y,K,[F],1)\cong\mathbb F$, so $K$ is fibered with fiber $F$ by  Theorem~\ref{thm:KnotFibered}.
\end{proof}


\section{Knots in torus bundles}\label{sect:TB}

In this section, we will prove Theorem~\ref{thm:TB}. Let $Y$ be a $T^2$--bundle over $S^1$, $K\subset Y$ be a null-homologous knot. Let $\gamma\subset Y\Bslash K$ be a loop which intersects a torus fiber exactly once transversely, and $\mu\subset Y\Bslash K$ be a meridian of $K$. As in Section~\ref{sect:ZeroMC}, the evaluation over $[\gamma]+[\mu]$ defines a ring homomorphism $\mathbb Z[H^1(\cdot)]\to \mathbb Z[\mathbb Z]=\mathbb Z[T,T^{-1}]$. Using this evaluation, we will get the corresponding twisted Heegaard Floer homology $\underline{HF^{\circ}}(\cdot;\Lambda)$ with coefficients in the Novikov ring $\Lambda$. The same argument as in \cite{OSzGenus} and \cite{NiFibred} shows that twisted knot Floer homology in this case detects the genus of a knot and whether a knot is fibered. See also \cite[Section~3.3]{NiNSSphere}. 

\begin{proof}[Proof of Theorem~\ref{thm:TB}]
Without loss of generality, we assume $K$ is nontrivial.
Let $\varphi\in H_2(Y,K)$ be a Seifert homology class, and let 
$$d=\max\left\{i\in\mathbb Z\left| \underline{\widehat{HFK}}(Y,K,\varphi,i)\ne0\right.\right\}.$$
Then $\underline{\widehat{HFK}}(Y,K,\varphi,d;\Lambda)$ is isomorphic to the mapping cone of
\[
\underline{v_{d-1}^+}\co \underline{A^+_{d-1}}\to \underline{B^+},
\]
and
$\underline{HF^+}(Y_0(K),\widehat{\varphi},d-1;\Lambda)$ is isomorphic to the mapping cone of
\[
\underline{v_{d-1}^+}+\underline{h_{d-1}^+}\co \underline{A^+_{d-1}}\to \underline{B^+}.
\]
Here $\underline{A^+_{d-1}}$ and $\underline{B^+}$ are twisted chain complexes.

Since $([\gamma]+[\mu])\cdot[\text{torus fiber}]\ne0$,
by \cite{AiPeters},
\[
H_*(\underline{B^+})=\underline{HF^+}(Y;\Lambda)\cong\Lambda,
\]
which is a $1$--dimensional vector space over the field $\Lambda$. So $(\underline{v_{d-1}^+})_*$ is either surjective or $0$. 

If $(\underline{v_{d-1}^+})_*$ is surjective, the same argument of counting dimensions as in Case~1 of Proposition~\ref{prop:ZeroTorsEqual} shows that 
\[
\mathrm{rank}_{\Lambda}\underline{HF^+}(Y_0(K),\widehat{\varphi},d-1;\Lambda)\ge\mathrm{rank}_{\Lambda}\underline{\widehat{HFK}}(Y,K,\varphi,d;\Lambda).
\]
Hence Property G follows just as before.

If $(\underline{v_{d-1}^+})_*$ is $0$, then $(\underline{v_{0}^+})_*=0$ by (\ref{eq:vIm}). Since $(\underline{v_{0}^+})_*$ and $(\underline{h_{0}^+})_*$ have the same rank, we also have $(\underline{h_{0}^+})_*=0$. Thus $(\underline{h_{d-1}^+})_*=0$ by (\ref{eq:hIm}).
So both mapping cones are quasi-isomorphic to $\underline{A^+_{d-1}}\oplus\underline{B^+}$, and Property G still follows.
\end{proof}

\setcounter{section}{0}
\renewcommand{\thesection}{\Alph{section}}

\section{Appendix: A proof of Theorem~\ref{thm:OriginalMC}}

In this appendix, we will sketch a proof of the version of Theorem~\ref{thm:OriginalMC} with $\mathbb F_2=\mathbb Z/2\mathbb Z$ coefficients, following the strategy of the proof of \cite[Theorem~4.5]{OSzBrDoubCov}. The reason that the proof only works over $\mathbb F_2$ is that the key lemma in homological algebra (Lemma~\ref{lem:HomAlg}) is only stated over $\mathbb F_2$. As remarked in a footnote in \cite[Section~4]{OSzBrDoubCov}, the proof can be carried over to $\mathbb Z$ coefficients routinely. However, since $\mathbb F_2$ coefficients are enough for the applications in this paper, we are satisfied with the current version.

The following lemma in homological algebra is \cite[Lemma~4.2]{OSzBrDoubCov}.

\begin{lem}\label{lem:HomAlg}
Let $\{A_i\}_{i=1}^{\infty}$ be a collection of chain complexes of $\mathbb F_2$--vector spaces and let
\[\{f_i:A_i\to A_{i+1}\}_{i=1}^{\infty}\]
be a collection of chain maps satisfying the following two properties:\newline
(1) $f_{i+1}\circ f_i$ is chain homotopically trivial, by a chain homotopy 
\[H_i: A_i\to A_{i+2},\]
(2) the map
\[
\phi_i=f_{i+2}\circ H_i+H_{i+1}\circ f_i: A_i\to A_{i+3}
\]
is a quasi-isomorphism.
\newline
Then the map
\[
\psi_i: A_{i}\to MC(f_{i+1})
\]
defined by
\[\psi_i(a_i)=(f_{i}(a_{i}),H_i(a_i))\]
is a quasi-isomorphism.
\end{lem}

Suppose that we have a pointed Heegaard $(n+1)$--tuple 
\[
(\Sigma, \mbox{\boldmath${\xi}$}^0,\mbox{\boldmath${\xi}$}^1,\dots,\mbox{\boldmath${\xi}$}^n,z)
\]
satisfying certain admissibility conditions. There is a standard way to define a map
\[
\mu_n:\bigotimes_{i=1}^nCF^+(\Sigma,\mbox{\boldmath${\xi}$}^{i-1},\mbox{\boldmath${\xi}$}^i,z)\to CF^+(\Sigma,\mbox{\boldmath${\xi}$}^{0},\mbox{\boldmath${\xi}$}^n,z)
\]
by counting pseudo-holomorphic $(n+1)$--gons,
where $CF^+(\Sigma,\mbox{\boldmath${\xi}$}^i,\mbox{\boldmath${\xi}$}^j,z)$ is the Heegaard Floer chain complex constructed for the Heegaard diagram $(\Sigma,\mbox{\boldmath${\xi}$}^i,\mbox{\boldmath${\xi}$}^j,z)$. These maps satisfy a well-known generalized associativity property
\begin{equation}\label{eq:Ainfty}
\sum_{i+j=n+1}\sum_{\ell=1}^{n-j+1}\mu_i(a_1\otimes\cdots\otimes a_{\ell-1}\otimes \mu_j(a_{\ell}\otimes\cdots\otimes a_{\ell+j-1})\otimes a_{\ell+j}\otimes\cdots\otimes a_n)=0.
\end{equation}

Let 
$$(\Sigma,\mbox{\boldmath${\alpha}$},\mbox{\boldmath${\beta}$},\mbox{\boldmath${\gamma}$}, \mbox{\boldmath${\delta}$}, w),$$
be the genus $h$ Heegaard quadruple constructed for the pair $(Y,K)$ in Section~\ref{sect:ZeroMC}.
We will consider a sequence of $h$--tuples of attaching curves $\{\mbox{\boldmath${\eta}$}^{i}\}_{i=1}^{\infty}$, where $\mbox{\boldmath${\eta}$}^{1}=\mbox{\boldmath${\beta}$}$ and $\mbox{\boldmath${\eta}$}^{3i+1}$ ($i>0$) consists of small Hamiltonian translates of curves in $\mbox{\boldmath${\beta}$}$, $\mbox{\boldmath${\eta}$}^{2}=\mbox{\boldmath${\gamma}$}$ and $\mbox{\boldmath${\eta}$}^{3i+2}$ ($i>0$) consists of small Hamiltonian translates of curves in $\mbox{\boldmath${\gamma}$}$, $\mbox{\boldmath${\eta}$}^{3}=\mbox{\boldmath${\delta}$}$ and $\mbox{\boldmath${\eta}$}^{3i+3}$ ($i>0$) consists of small Hamiltonian translates of curves in $\mbox{\boldmath${\delta}$}$. Let
\[
\mbox{\boldmath${\eta}$}^{i}=(\eta_1^{i},\dots,\eta_h^{i}).
\]

Note that $(\Sigma,\mbox{\boldmath${\eta}$}^{3i+1},\mbox{\boldmath${\eta}$}^{3i+2},w)$ is a Heegaard diagram for $\#^{h-1}S^1\times S^2$, $i\ge0$.
Let $\Theta_{3i+1,3i+2}$ be the top generator of $\widehat{CF}(\Sigma,\mbox{\boldmath${\eta}$}^{3i+1},\mbox{\boldmath${\eta}$}^{3i+2},w)$. 
Similarly, define $\Theta_{3i+3,3i+4}$ to be  the top generator of $\widehat{CF}(\Sigma,\mbox{\boldmath${\eta}$}^{3i+3},\mbox{\boldmath${\eta}$}^{3i+4},w)$, $i\ge0$.

The diagram
$(\Sigma,\mbox{\boldmath${\eta}$}^{3i+2},\mbox{\boldmath${\eta}$}^{3i+3},w)$ is a Heegaard diagram for $(\#^{h-1}S^1\times S^2)\#L(n,1)$, $i\ge0$. There are $n$ generators of $\widehat{CF}(\Sigma,\mbox{\boldmath${\eta}$}^{3i+2},\mbox{\boldmath${\eta}$}^{3i+3},w)$ which can be viewed as ``top'' generators, but only one of them corresponds to $\mathfrak t_k$. Let this generator be $\Theta_{3i+2,3i+3}$. (More precisely, the pair of pants construction gives us a cobordism 
\[X_{\alpha,\eta^{3i+2},\eta^{3i+3}}:Y_0(K)\sqcup (\#^{h-1}S^1\times S^2)\#L(n,1) \to Y_n(K).\] Then the Spin$^c$ structure associated with $\Theta_{3i+2,3i+3}$ is the restriction of the Spin$^c$ structure over $X_{\alpha,\eta^{3i+2},\eta^{3i+3}}$ that extends $\mathfrak t_k$.)

Now we define the map
\[
f_i: CF^+(\Sigma,\mbox{\boldmath${\alpha}$},\mbox{\boldmath${\eta}$}^{i},w)\to CF^+(\Sigma,\mbox{\boldmath${\alpha}$},\mbox{\boldmath${\eta}$}^{i+1},w)
\]
by 
\[
f_i(x)=\mu_2(x,\Theta_{i,i+1}).
\]
It is a chain map by (\ref{eq:Ainfty}) and the fact that
\begin{equation}\label{eq:mu1theta}
\mu_1(\Theta_{i,i+1})=0.
\end{equation}

We also define
\[
H_i: CF^+(\Sigma,\mbox{\boldmath${\alpha}$},\mbox{\boldmath${\eta}$}^{i},w)\to CF^+(\Sigma,\mbox{\boldmath${\alpha}$},\mbox{\boldmath${\eta}$}^{i+2},w)
\]
by
\[
H_i(x)=\mu_3(x,\Theta_{i,i+1},\Theta_{i+1,i+2}).
\]
It is standard to check 
\begin{equation}\label{eq:fComb}
f_{i+1}\circ f_i=\partial\circ H_i+H_{i}\circ\partial
\end{equation}
using (\ref{eq:Ainfty}), (\ref{eq:mu1theta}) and the fact that
\begin{equation}\label{eq:mu2theta}
\mu_2(\Theta_{i,i+1},\Theta_{i+1,i+2})=0.
\end{equation}

We then define
\[
G_i: CF^+(\Sigma,\mbox{\boldmath${\alpha}$},\mbox{\boldmath${\eta}$}^{i},w)\to CF^+(\Sigma,\mbox{\boldmath${\alpha}$},\mbox{\boldmath${\eta}$}^{i+3},w)
\]
by
\[
G_i(x)=\mu_4(x,\Theta_{i,i+1},\Theta_{i+1,i+2},\Theta_{i+2,i+3}).
\]
Let  
\[
\sigma: CF^+(\Sigma,\mbox{\boldmath${\alpha}$},\mbox{\boldmath${\eta}$}^{i},w)\to CF^+(\Sigma,\mbox{\boldmath${\alpha}$},\mbox{\boldmath${\eta}$}^{i+3},w)
\]
be the map
defined by \[\sigma(x)=\mu_2(x,\mu_3(\Theta_{i,i+1},\Theta_{i+1,i+2},\Theta_{i+2,i+3})).\]
By (\ref{eq:mu1theta}), $\Theta_{i,i+1},\Theta_{i+1,i+2},\Theta_{i+2,i+3}$ are all cycles. It follows from (\ref{eq:Ainfty}) and (\ref{eq:mu2theta}) that $\mu_3(\Theta_{i,i+1},\Theta_{i+1,i+2},\Theta_{i+2,i+3})$ is a cycle. So $\sigma$ is a chain map.
Using (\ref{eq:Ainfty}), (\ref{eq:mu1theta}) and (\ref{eq:mu2theta}),
we get
\begin{equation}\label{eq:FiveTerms}
\partial\circ G_i+G_{i}\circ\partial+f_{i+2}\circ H_i+H_{i+1}\circ f_{i}+\sigma=0.
\end{equation}

We claim that
\begin{equation}\label{eq:mu3theta}
\mu_3(\Theta_{i,i+1},\Theta_{i+1,i+2},\Theta_{i+2,i+3})=\Theta_{i,i+3}+Uy,
\end{equation}
where $\Theta_{i,i+3}$ is the top generator of $\widehat{CF}(\Sigma,\mbox{\boldmath${\eta}$}^{i},\mbox{\boldmath${\eta}$}^{i+3},w)$, and $y$ is some element in $CF^+(\Sigma,\mbox{\boldmath${\eta}$}^{i},\mbox{\boldmath${\eta}$}^{i+3},w)$. 

Since $\mbox{\boldmath${\eta}$}^{i+3}$ is a Hamiltonian translate of $\mbox{\boldmath${\eta}$}^{i}$, $\mu_2(\cdot,\Theta_{i,i+3})$ is a chain homotopy equivalence. 
The claim implies that 
\[
\sigma=\iota+U\rho,
\]
where $\iota$ is a chain homotopy equivalence, and $U\rho$ is a chain map. (It is not clear whether $\rho$ is a chain map.)

Let 
\[\kappa: CF^+(\Sigma,\mbox{\boldmath${\alpha}$},\mbox{\boldmath${\eta}$}^{i+3},w)\to CF^+(\Sigma,\mbox{\boldmath${\alpha}$},\mbox{\boldmath${\eta}$}^{i},w)\] be a chain map which is an inverse to $\iota$ up to chain homotopy. 
Consider the chain map
\begin{eqnarray*}
\kappa'&=&\kappa(\mathrm{id}-U\rho\kappa+(U\rho\kappa)^2-(U\rho\kappa)^3+\cdots)\\
&=&(\mathrm{id}-\kappa U\rho+(\kappa U\rho)^2-(\kappa U\rho)^3+\cdots)\kappa.
\end{eqnarray*}
For any $x\in CF^+(\Sigma,\mbox{\boldmath${\alpha}$},\mbox{\boldmath${\eta}$}^{i+3},w)$, $U^Nx=0$ when $N$ is sufficiently large. So $\kappa'$ is well-defined,
and we can compute
\[
(\iota+U\rho)\kappa'\simeq\mathrm{id},\quad \kappa'(\iota+U\rho)\simeq\mathrm{id}.
\] 
So $\sigma=\iota+U\rho$ is a quasi-isomorphism.
Now (\ref{eq:FiveTerms}) implies that \[f_{i+2}\circ H_i+H_{i+1}\circ f_{i}\] is chain homotopic to a quasi-isomorphism, so it is a quasi-isomorphism as well. This finishes checking Condition~(2) in Lemma~\ref{lem:HomAlg}. Condition~(1) in Lemma~\ref{lem:HomAlg} is just (\ref{eq:fComb}).

To finish the proof of Theorem~\ref{thm:OriginalMC}, we only need to prove (\ref{eq:mu3theta}), which is essentially \cite[Equation~(11)]{OSzBrDoubCov}. In fact, as in \cite{OSzBrDoubCov}, the proof reduces to the computation in a genus--$1$ surface. 

\begin{figure}[ht]
\begin{picture}(340,200)
\put(70,0){\scalebox{0.51}{\includegraphics*
{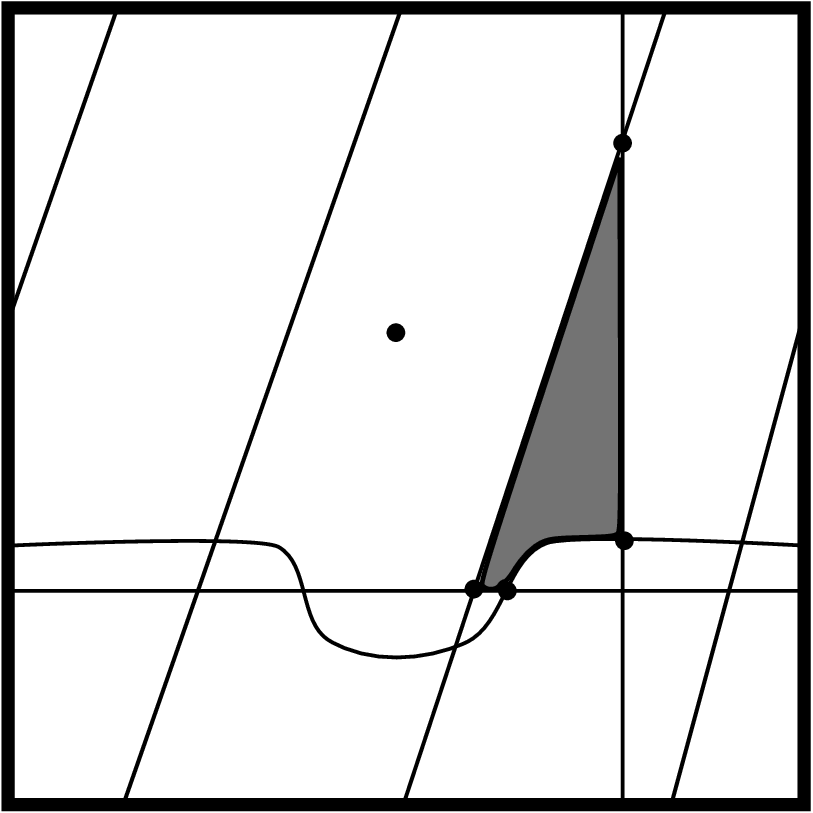}}}

\put(90,44){$\eta^2$}

\put(125,100){$\eta^3$}

\put(225,110){$\eta^4$}

\put(90,72){$\eta^5$}

\put(170,111){$w$}

\put(171,60){$\scriptstyle\Theta_{2,3}$}

\put(194,46){$\scriptstyle\Theta_{2,5}$}

\put(225,160){$\scriptstyle\Theta_{3,4}$}

\put(224,71){$\scriptstyle\Theta_{4,5}$}

\end{picture}
\caption{The domain of a holomorphic quadrilateral}\label{fig:Quad}
\end{figure}

There are two small triangles not containing $w$ which contribute to \[\mu_2(\Theta_{i,i+1},\Theta_{i+1,i+2}).\] (These two triangles cancel, which is part of the reason for (\ref{eq:mu2theta}).) The curve $\eta^{i+3}$ cuts exactly one of these two triangles so as to get a positive quadrilateral with $\Theta_{i,i+1}$ and $\Theta_{i+1,i+2}$ being two of its vertices. The domain of this quadrilateral when $i=2$ is shown in Figure~\ref{fig:Quad}. 
All other quadrilaterals contributing to \[\mu_3(\Theta_{i,i+1},\Theta_{i+1,i+2},\Theta_{i+2,i+3})\] necessarily contain the base point $w$. So we have (\ref{eq:mu3theta}).

\end{document}